\documentclass[11pt]{amsart}

\usepackage{amsfonts,amssymb,enumerate,color,bm,hhline,makecell,array}

\usepackage{array}
\usepackage{mathtools,amssymb}
\usepackage{mathrsfs}
\usepackage{comment}
\usepackage[all,ps,cmtip]{xy} 
\usepackage[normalem]{ulem}
\usepackage[colorlinks=true,citecolor=, urlcolor=, linkcolor=, pagebackref]{hyperref}
\usepackage{tikz-cd}
\usepackage{amscd}
\usepackage[shortlabels]{enumitem}
\usepackage{tensor}
\usepackage{tikz}
\usetikzlibrary{matrix,arrows}
\usepackage{bm}
\usepackage{graphicx}
\usepackage{extarrows}
\usepackage{subcaption}
\usepackage{mathtools}
\usepackage{comment}
\usetikzlibrary{cd, calc}

\newtheorem{theorem}{Theorem}

\newcommand{\To}{\longrightarrow}

\newtheorem{prop}{Proposition}[section]
\newtheorem{lemma}{Lemma}[section]

\newtheorem{coro}{Corollary}

\newtheorem{rmk}{Remark}[section]

\newcommand{\At}{\operatorname{At}}
\newcommand{\Atp}{\operatorname{At_{p}}}

\newcommand{\bl}{\operatorname{Bl}}
\newcommand{\Spec}{\operatorname{Spec}}

\newcommand{\Pic}{\operatorname{Pic}}
\newcommand{\ext}{\operatorname{Ext}}
\newcommand{\Hom}{\operatorname{Hom}}
\newcommand{\ku}{\operatorname{Kum}}
\newtheorem{thm}{Theorem}[section]

\usepackage{mathrsfs}
\newcommand{\Tot}{\operatorname{Tot}}
\newcommand{\sym}{\operatorname{Sym}}

\newcommand{\td}{\operatorname{td}}

\title{Moduli of sheaves and deformation to the normal cone}
\author{Yifan Zhao}
\address{Department of Mathematics, Imperial College, London SW7 2AZ, United Kingdom}
\email{yz23@ic.ac.uk}

\date{}
\begin{document}
\maketitle
\begin{abstract}
Given a closed immersion between arbitrary smooth complex projective varieties, we prove that the two operations: (1) taking the moduli space of stable sheaves, and (2) taking the deformation to the normal cone, commute in a precise sense. In the case of curves inside symplectic surfaces, previously studied by Donagi--Ein--Lazarsfeld, the corresponding deformation to the normal cone space is an open subset of the relative moduli space of sheaves. As an application, we show that generalized Kummer varieties degenerate to natural symplectic subvarieties of the Hitchin system for curves of genus $g \geq 2$.
\end{abstract}

\section{Introduction}
Given a closed immersion of complex schemes $i:Y\xhookrightarrow{}X,$ the associated deformation to the normal cone is the following scheme over the affine line $\mathbb{C}:$ $$
D_{Y/X}:=
\bl_{Y\times 0}(X\times \mathbb{C}) \backslash  \bl_{Y\times 0}(X\times0) \To \mathbb{C}.
$$
Its generic fibre over $t\neq0$ is $X$ and the special fibre over $0\in \mathbb{C}$ is the normal cone $C_{Y/X}.$

Assume $X,Y$ are projective and smooth, $C_{Y/X}$ is the normal bundle $N_{Y/X}$. The central object of our study is the relative moduli space, introduced in \cite{sim}, of Gieseker stable sheaves properly supported on the fibres of $D_{Y/X}\to \mathbb{C}$.

Fix a stable sheaf $E_0$ on $Y,$ the moduli space $\mathcal{M}_{Y}:=\mathcal{M}_{Y}([E_0])$ parametrizes stable sheaves numerically equivalent to $E_0.$ Similarly,  we denote by $$\mathcal{M}_X, \mathcal{M}_{N_{Y/X}}, \text{and } \mathcal{M}_{D_{Y/X}/\mathbb{C}}$$ the (relative) moduli spaces parametrizing sheaves numerically equivalent to $i_*E_0$ on $X, N_{Y/X}$ and the fibres of $D_{Y/X}\to \mathbb{C}$ respectively. Pushing stable sheaves forward via $i$ induces a closed immersion $\mathcal{M}_{Y}\xhookrightarrow{}\mathcal{M}_X.$

We assume throughout that there are no strictly semi-stable sheaves. Our main result says the two constructions (1) taking (relative) moduli spaces of Gieseker stable sheaves, and (2) deformation to the normal cone, commute:
\begin{theorem}[Theorem \ref{thm3.1}] \label{thm1} 
There exists an injective morphism of schemes over the affine line
\begin{equation}
\psi: \label{eq1}
D_{\mathcal{M}_Y/\mathcal{M}_X} \To \mathcal{M}_{D_{Y/X}/\mathbb{C}},
\end{equation}
extending the isomorphism between the two families over $\mathbb{C}^* \subset \mathbb{C}.$

If moreover  open subsets $U\subset \mathcal{M}_{D_{Y/X}/\mathbb{C}}$ and $\psi^{-1}(U)\subset D_{\mathcal{M}_Y/\mathcal{M}_X}$ are varieties and $U$ is normal, the restriction $\psi:\psi^{-1}(U)\to U$ is an open immersion.
\end{theorem}
Key ingredients of our proof include the Rees and relative Spec construction. In Section \ref{sec2}, we construct a flat $D_{\mathcal{M}_Y/\mathcal{M}_X}$-family of sheaves $\mathcal{F}$ on fibres of $D_{Y/X}\to\mathbb{C},$ which induces $\psi: D_{\mathcal{M}_Y/\mathcal{M}_X}\to \mathcal{M}_{D_{Y/X}/\mathbb{C}}$ in  (\ref{eq1}).  

\begin{equation*}
\begin{tikzcd}
    \text{Sheaf $\mathcal{F}$ on $D_{Y/X}$}\times_{\mathbb{C}}D_{\mathcal{M}_Y/\mathcal{M}_X} \arrow{rrr}{\text{Classifying map}}  &  &  & \text{The morphism $\psi$} \\
    \\
    \text{Universal sheaf F on $X\times \mathcal{M}_{X}$} \arrow{uu}{\text{Relative Spec construction}}
  \end{tikzcd}
\end{equation*}

Classically, Donagi--Ein--Lazarsfeld \cite{DEL} studied such relative moduli spaces associated with curves on K3 surfaces: $C\xhookrightarrow{i}{}S$, where $C$ has genus $g \geq 2$. The Chern character $ch(E_0)=(r,d)\in H^0(C)\oplus H^2(C)$ satisfies $r>0, \gcd(r,d)=1.$

In this setting, $\mathcal{M}_{D_{C/S}/\mathbb{C}}$ is smooth \cite{dCMS}. The moduli of stable bundles $\mathcal{M}_{C}$ embeds in the hyperkähler manifold $\mathcal{M}_S$ as a Lagrangian submanifold. It is argued in \cite{DEL} that the $\mathbb{C}$-family $\mathcal{M}_{D_{C/S}/\mathbb{C}}$ is a degeneration of the Beauville--Mukai system \cite{BM} to the Hitchin system \cite{hit}: indeed, the special fibre $\mathcal{M}_{N_{C/S}}$ over $0\in \mathbb{C}$, parametrizing stable sheaves on the normal bundle $$
N_{C/S} \cong T^*C,
$$ is isomorphic to the moduli space $\mathcal{M}_{\text{GL(r)}}$ of stable $\text{GL(r)}$ Higgs bundles, the total space of the Hitchin integrable system.

The degeneration $\mathcal{M}_{D_{C/S}/\mathbb{C}}\to \mathbb{C}$ thus connects two classes of extensively studied Lagrangian-fibred hyperkähler manifolds. More recently, de Cataldo--Maulik--Shen \cite{dCMS} exploited this degeneration (replacing K3 surfaces by abelian surfaces) to prove the P=W conjecture \cite{p=w} for genus 2 curves. Theorem \ref{thm1} provides a surprisingly simple description of $\mathcal{M}_{D_{C/S}/\mathbb{C}}\to \mathbb{C}$:
\begin{coro}\label{coro1}
Suppose $C$ is a smooth curve of genus $g \geq 2$ in a K3 or abelian surface $S,$ $D_{\mathcal{M}_C/\mathcal{M}_S}$ is an open dense subset of $\mathcal{M}_{D_{C/S}/\mathbb{C}}$. 

The complement of $D_{\mathcal{M}_C/\mathcal{M}_S}$ in $\mathcal{M}_{D_{C/S}/\mathbb{C}}$ is precisely the closed subset of $\mathcal{M}_{\emph{GL(r)}}$ parametrizing Higgs bundles with unstable underlying bundles. 
\end{coro}

We deduce from Corollary \ref{coro1} that:
\begin{coro}[Proposition \ref{prop4.1}]
There exists a Poisson structure on $\mathcal{M}_{D_{C/S}/\mathbb{C}}$ whose symplectic leaves are the fibres of $  \mathcal{M}_{D_{C/S}/\mathbb{C}}\to \mathbb{C},$ i.e. the Beauville--Mukai system $\mathcal{M}_{S}$ and the Hitchin system $\mathcal{M}_{\emph{GL(r)}}.$
\end{coro}
The relative symplectic/Poisson structure is also obtained by Franco \cite{franco} using different methods.

Let us now turn to the case of curves on abelian surfaces: $C\xhookrightarrow{i}A$. $\mathcal{M}_{A}$ is a holomorphic symplectic manifold, but is decomposable in the sense of Beauville--Bogomolov \cite{bby}. Debarre \cite{kum} showed there exists a codimension $4$ smooth subvariety $\ku$ inside $\mathcal{M}_{A}$, which is an irreducible symplectic (equivalently, hyperkähler) variety of generalized Kummer type.

It is natural to ask whether $\ku$ also degenerates to a subsystem of the GL(r) Hitchin integrable system for the curve $C$ of genus $g \geq 2$. We answer this affirmatively  by studying the (Hamiltonian) $\Pic^0(A)$ action on the moduli space of Higgs bundles $\mathcal{M}_{\text{GL(r)}}$, which is induced by the pull-back map: $\Pic^0(A)\xrightarrow{i^*} \Pic^0(C)$.

Denote the moment map of this $\Pic^0(A)$ action by $\mu$, which will be given explicitly in Section \ref{sec4} as the morphism (\ref{eq8}). The holomorphic symplectic reduction of $\mathcal{M}_{\text{GL(r)}}$, denoted by $\mathcal{N}^{\vee}$, is a symplectic variety with orbifold singularities.

Moreover, consider the morphism $$\tilde{\textbf{c}}_{1}:\mathcal{M}_{\text{GL(r)}} \xrightarrow[]{} \text{CH}_0(A)_0\cong A, (\mathcal{E},\phi)\longmapsto i_*c_1(\mathcal{E})-i_*c_1(E_0),$$ where $\mathcal{E}
$ is the underlying bundle and $E_0$ is as before a fixed stable bundle on $C.$ The map $\tilde{\textbf{c}}_{1}$ is an isotrivial fibration and its fibres give slices of the $\Pic^0(A)$ action. We denote the closed codimension $4$ subvariety $\tilde{\textbf{c}}_{1}^{-1}(0)\cap \mu^{-1}(0)\subset \mathcal{M}_{\text{GL(r)}}$ by $\mathcal{N}.$

\begin{theorem}[Proposition \ref{prop4.2}, \ref{prop4.4}]
The hyperkähler submanifold $\ku \subset \mathcal{M}_{A}$ degenerates to the smooth Lagrangian-fibred symplectic subvariety $\mathcal{N}$ in the moduli space $\mathcal{M}_{\emph{GL(r)}}$ of $\emph{GL(r)}$ stable Higgs bundles.

The symplectic subvariety $\mathcal{N}$ and the symplectic reduction $\mathcal{N}^{\vee}$ are dual Lagrangian fibrations, and there exists a finite surjective map: $\mathcal{N}\to \mathcal{N}^{\vee}.$  
\end{theorem}
It is observed by de Cataldo--Maulik--Shen \cite{sl} (see also \cite{franco}) that when $C\xhookrightarrow{}A$ is the Abel-Jacobi map of a genus 2 curve, $\mathcal{N}$ is precisely the moduli space of $\text{SL(r)}$ Higgs bundles on $C$. Its dual fibration $\mathcal{N}^{\vee}$ is the moduli of $\text{PGL(r)}$ Higgs bundles.

In the higher dimensional cases, as pointed out in \cite{DEL}, an interesting example is when $X$ is a hyperkähler manifold and $Y$ is a Lagrangian submanifold. The aforementioned Beauville--Mukai system is then replaced by the moduli space of Lagrangian sheaves \cite[Section 8]{DM} on $X$, while the Hitchin system is replaced by Simpson's moduli space of Higgs bundles \cite{sim} on $Y$.
\subsection*{An example: K3-fibred threefolds and 1-rigid classes}
One may wonder if the injective morphism $\psi$ in  (\ref{eq1}) is always an open immersion. The following example shows that in general this hope fails.

Let $X$ be a K3-fibred threefold over a smooth curve $C$ and $Y\xhookrightarrow{i}X$ be the smooth fibre over $0\in C.$ Suppose moreover the Chern character is
$$ 
ch=(0,\beta,1)\in H^{0}(Y)\oplus H^{2}(Y)\oplus H^{4}(Y),
$$ and that any curve on $X$ with fundamental class $i_*\beta$ is integral and is contained in $Y\subset X.$ Such (effective) $\beta$ can be chosen from 1-rigid classes with respect to the smooth family $X\to C,$ in the sense of \cite[Section 0.6]{kkv}.

Since curves with class $i_*\beta$ are integral and can not move outside $Y,$ the stable sheaves on $X$ are scheme-theoretically supported on the K3 surface $Y,$ making $\mathcal{M}_{X}=\mathcal{M}_{Y}$ smooth. $\mathcal{M}_{D_{Y/X}/\mathbb{C}}$ has two irreducible components, one parametrizing sheaves on $X\times \mathbb{C}^*\subset D_{Y/X}$ and the other parametrizing sheaves on the special fibre $Y\times \mathbb{C}.$ The two components have the same dimension and intersect nontrivially. The morphism $\psi$ maps bijectively to the first irreducible component and is not an open immersion.

In general, whenever $\mathcal{M}_X$ is a variety, $D_{\mathcal{M}_Y/\mathcal{M}_X}$ is also a variety \cite[II, 7.16]{hartshorn}. The image of $\psi$
is then contained in a single irreducible component of the relative moduli $\mathcal{M}_{D_{Y/X}/\mathbb{C}}.$

\subsection*{Notation and convention}
We denote (possibly twisted) universal sheaves on $X\times \mathcal{M}_X$ and $Y\times \mathcal{M}_Y$ by $F$ and $E$ respectively. We choose a generic polarization $\mathcal{O}(1)$ on $X,$ and $D_{Y/X}$ is polarized by the pull-back of $\mathcal{O}(1)$ via $D_{Y/X}\subset \bl_{Y\times 0}(X\times \mathbb{C})\to X.$

The ideal sheaf of the smooth closed subscheme $Y\xhookrightarrow{i}X$ is $I$ and the ideal sheaf of $\mathcal{M}_Y\xhookrightarrow{j}\mathcal{M}_X$ is $J.$ The projection $A\times B\to A$ is  $\pi_A$. We use the same notation for a morphism and its base change, and omit the pull-back symbol when it is clear from the context. The derived dual $R\mathcal{H}om(\mathcal{L},\mathcal{O})$ of a complex $\mathcal{L}$ is denoted by $\mathcal{L}^{\vee},$ and the underived dual of a sheaf $L$ is $L^*.$ 
\subsection*{Acknowledgement}
I thank my supervisors Richard Thomas and Soheyla Feyzbakhsh for many helpful discussions and suggestions.

I also would like to thank Riccardo Carini, Daniel Huybrechts, Mirko Mauri, Tony Pantev and Junliang Shen for related conversations and comments, as well as the organizers of the conferences “Algebraic Geometry: Wall Crossing and Moduli Spaces, Varieties and Derived Categories" and “Stratifications of Higgs bundle moduli spaces and related topics" where parts of the work were presented.

This work was financially supported by the Engineering and Physical Sciences Research Council [EP/S021590/1], the EPSRC Centre for Doctoral Training in
Geometry and Number Theory (The London School of Geometry and Number
Theory),  University College London.
\section{Deformation to the normal cone}
\label{sec2}
In this section, we construct the morphism $\psi$
in Theorem \ref{thm1}. We first recall the description of the deformation to the normal cone via Rees algebras. For more discussions on Rees algebras, deformation to the normal cone and applications to intersection theory, see \cite[Section 2.3]{CG} and \cite[Section 5.1]{ful}.

Set $I^i$ to be $\mathcal{O}_{X}$ if $i \leq 0,$ the Rees algebra is by definition $\oplus_{n\in \mathbb{Z}} t^n I^{-n}.$ Consider the relative spectrum
$$D_{Y/X,\text{Rees}}:=\Spec_{X} (\oplus_{n\in \mathbb{Z}} t^n I^{-n} ).
$$ The inclusion $$\mathcal{O}_{X}\otimes \mathbb{C}[t] \To \oplus_{n\in \mathbb{Z}} t^n I^{-n} $$ makes $\oplus_{n\in \mathbb{Z}} t^n I^{-n} $ a quasi-coherent sheaf on $X$ of $\mathcal{O}_{X}[t]$ algebras and gives a morphism $$D_{Y/X,\text{Rees}} \To X\times \mathbb{C}.$$ 

\begin{lemma}[{\cite[Section 5.1]{ful}}]
The above morphism: $D_{Y/X, \emph{Rees}}\to X\times \mathbb{C}$ is identical to the blow-down map $$
D_{Y/X}\subset \bl_{Y\times 0}(X\times \mathbb{C})\To X\times \mathbb{C}.
$$
In particular, $D_{Y/X}\to X\times \mathbb{C}$ is affine.
\end{lemma}
The $\mathbb{Z}$-grading on $\oplus_{n\in \mathbb{Z}} t^n I^{-n}$ induces a $\mathbb{C}^*$ action on $D_{Y/X}$. The variable $t$ can be viewed as the $\mathbb{C}^*$ equivariant parameter and $t^n I^{-n}$ is the sheaf of equivariant regular functions on $D_{Y/X}$ with weight $n$.

Now we construct a family of stable sheaves $\mathcal{F}$ on fibres of $
D_{Y/X}\to \mathbb{C},$ which is flat over the base $D_{\mathcal{M}_Y/\mathcal{M}_X}$ and induces the classifying map
$$
\psi:D_{\mathcal{M}_Y/\mathcal{M}_X} \To \mathcal{M}_{D_{Y/X}/\mathbb{C}}.
$$

\begin{prop} \label{prop2.1}
Let $F$ be a universal sheaf \footnote{Strictly speaking, $F$ is a twisted sheaf on $X\times \mathcal{M}_{X}$ for some Brauer class $\alpha\in H^2(\mathcal{O}^*_{X\times \mathcal{M}_{X}})$ pulled back from $\mathcal{M}_{X}$, but throughout the paper we simply call $F$ a sheaf.} on $X\times \mathcal{M}_{X}$, there exists a flat $D_{\mathcal{M}_Y/\mathcal{M}_X}$-family of stable sheaves $\mathcal{F}$ on the fibres of $D_{Y/X}\to \mathbb{C}$, extending the pull-back of $F$ to $$X\times \mathcal{M}_{X}\times \mathbb{C}^*\subset D_{Y/X}\times_{\mathbb{C}}D_{\mathcal{M}_Y/\mathcal{M}_X}.$$ $\mathcal{F}$ induces the morphism $\psi$ in \emph{(\ref{eq1})}.
\end{prop}
\begin{proof}
Recall that $J$ is the ideal sheaf of $\mathcal{M}_{Y}$ in $\mathcal{M}_{X}$, we know $$
D_{\mathcal{M}_Y/\mathcal{M}_X} =\Spec_{\mathcal{M}_{X}}(\oplus_{n\in \mathbb{Z}} t^n J^{-n})
$$ and $$
D_{Y/X}\times_{\mathbb{C}}D_{\mathcal{M}_Y/\mathcal{M}_X}=\Spec_{X\times \mathcal{M}_{X}}\bigl((\oplus _{n\in \mathbb{Z}}t^{-n}I^n) \otimes_{\mathbb{C}[t]} (\oplus _{m\in \mathbb{Z}} t^{-m}J^m)\bigr).
$$

We now construct a sheaf $\mathcal{F}$ on $D_{Y/X}\times_{\mathbb{C}}D_{\mathcal{M}_Y/\mathcal{M}_X},$ which is flat over $D_{\mathcal{M}_Y/\mathcal{M}_X}$ and pushes forward to $$\oplus _{n\in \mathbb{Z}}t^{-n}J^n\otimes F$$ via the affine morphism $$
D_{Y/X}\times_{\mathbb{C}}D_{\mathcal{M}_Y/\mathcal{M}_X}\To X\times \mathcal{M}_X.
$$ 

The pull-back of $F$ via the closed immersion $X\times\mathcal{M}_Y \xhookrightarrow{j} X\times \mathcal{M}_X,$ which is $F/JF$, is a flat $\mathcal{M}_Y$-family of sheaves scheme-theoretically supported on $Y\subset X.$ The image of $IF \subset F\to F/JF$ is therefore zero and we have the morphism
\begin{equation}
\label{eq2}
I\otimes F \To JF\cong J\otimes F,
\end{equation} from which we get a multiplication action by the sheaf of algebras $$(\oplus _{n\in \mathbb{Z}}t^{-n}I^n) \otimes_{\mathbb{C}[t]} (\oplus _{m\in \mathbb{Z}} t^{-m}J^m)$$ on 
$\oplus _{n\in \mathbb{Z}}t^{-n}J^n\otimes F.$ The relative Spec construction then produces the desired sheaf $\mathcal{F}$ on $D_{Y/X}\times_{\mathbb{C}}D_{\mathcal{M}_Y/\mathcal{M}_X}=\Spec_{X\times \mathcal{M}_{X}}\bigl((\oplus _{n\in \mathbb{Z}}t^{-n}I^n) \otimes_{\mathbb{C}[t]} (\oplus _{m\in \mathbb{Z}} t^{-m}J^m)\bigr).$ $\mathcal{F}$ is automatically a flat $D_{\mathcal{M}_Y/\mathcal{M}_X}$-family of stable sheaves since $F$ is a flat $\mathcal{M}_X$-family of stable sheaves.
\end{proof}
\begin{rmk}
The push-forward of $\mathcal{F}$ via $$D_{Y/X}\times_{\mathbb{C}}D_{\mathcal{M}_Y/\mathcal{M}_X}\To X\times D_{\mathcal{M}_Y/\mathcal{M}_X}$$ is the pull-back of $F$ via $$X\times D_{\mathcal{M}_Y/\mathcal{M}_X}\To X\times \mathcal{M}_X.$$ 
\end{rmk}

\section{Spectral correspondence}
In this section, we show that the morphism $\psi$ constructed in Proposition \ref{prop2.1} is injective. Our argument relies on the key Proposition \ref{prop3.1}, whose proof is deferred to Section \ref{sec5}.

Note that $\psi$ is a morphism of schemes over the affine line $\mathbb{C}.$ It restricts to an isomorphism on the open subsets over $\mathbb{C}^*\subset \mathbb{C}.$ On the special fibre over $0\in \mathbb{C},$ $\psi$ becomes 
\begin{equation*}
C_{\mathcal{M}_Y/\mathcal{M}_X}\To \mathcal{M}_{N_{Y/X}},  
\end{equation*}
where $C_{\mathcal{M}_Y/\mathcal{M}_X}$ is the normal cone of $\mathcal{M}_Y$ in $\mathcal{M}_X.$ We apply the spectral correspondence (as used in the theory of Higgs bundles) to describe the stable sheaves on $N_{Y/X}$ which correspond to the image of $C_{\mathcal{M}_Y/\mathcal{M}_X}\to \mathcal{M}_{N_{Y/X}}.$  
\begin{lemma}[Spectral correspondence]
Let $V\xrightarrow[]{\pi}Y$ be a vector bundle over a smooth polarized variety $\bigl(Y,\mathcal{O}_{Y}(1)\bigr)$. A coherent sheaf $G$ on $V$ with compact support corresponds to a unique Higgs pair $(\pi_*G,\phi)$, where $\phi$ is the Higgs field:$$
\pi_*G \To \pi_*G\otimes V
$$ induced by the multiplications of $\mathcal{O}_{V}$ on $G.$

Moreover, $G$ is stable with respect to $\pi^*\mathcal{O}_{Y}(1)$ if $\pi_*G$ is stable with respect to $\mathcal{O}_{Y}(1).$
\end{lemma}
\begin{proof}
The space $V=\Spec_Y(\sym(V^*))$ and the symmetric algebra is generated by $V^*$, therefore the Higgs field $\phi$ determines the multiplications of $\mathcal{O}_{V}$ on $G$.
\end{proof}
Given any closed point $x=(p,v)
\in C_{\mathcal{M}_Y/\mathcal{M}_X}$ where $p\in \mathcal{M}_Y$ and $v$ is a normal vector at $p,$ we denote the restriction of the sheaf $\mathcal{F}$ (constructed in Proposition \ref{prop2.1}) to $N_{Y/X}\times x$ by $\mathcal{F}_{x}.$ The normal vector $v$ is a morphism $$
v:(J/J^2)|_{p} \To \mathbb{C}
$$ since $J$ is the ideal sheaf of $\mathcal{M}_Y$ in $\mathcal{M}_X.$ 

Recall that we denote a universal sheaf on $Y\times \mathcal{M}_{Y}$ by $E$. Let $E_p$ be the stable sheaf on $Y$ corresponding to $p\in \mathcal{M}_Y.$ The morphism (\ref{eq2}) induces:
\begin{equation} \label{eq3}
I/I^2\otimes F \To J/J^2\otimes F,
\end{equation}
whose
composition with $v$ gives a Higgs field
\begin{equation}\label{eq4}
\phi_v: I/I^2\otimes E_{p} \To E_{p}.
\end{equation}
It is clear that:
\begin{lemma} \label{lem3.2}
Via the spectral correspondence, the sheaf $\mathcal{F}_{x}$ is identified with the Higgs pair $(E_p,\phi_v).$    
\end{lemma}
Restricting to the closed point $p\in\mathcal{M}_Y$, the map (\ref{eq3}) induces the morphism 
\begin{equation} \label{eq5}
\psi_{p}: \Hom_{\mathbb{C}}(J/J^2|_p,\mathbb{C})\To \Hom_{Y}(E_p,E_p\otimes N_{Y/X}),
\end{equation}
sending a normal vector $v$ to the Higgs field $\phi_v.$

\begin{prop}\label{prop3.1}
The map $\psi_{p}$ is injective.
\end{prop}
We defer the proof of this technical result to Section \ref{sec5}. We will equate $\psi_{p}$ with what we call the reduced Atiyah class $\Atp$, and conclude by showing that $\Atp$ is injective.

\begin{thm}[Theorem \ref{thm1}] \label{thm3.1}
The morphism \emph{$\psi$} in \emph{(\ref{eq1}):} $D_{\mathcal{M}_Y/\mathcal{M}_X} \to \mathcal{M}_{D_{Y/X}/\mathbb{C}}$ is set-theoretically injective.

If moreover  open subsets $U\subset \mathcal{M}_{D_{Y/X}/\mathbb{C}}$ and $\psi^{-1}(U)\subset D_{\mathcal{M}_Y/\mathcal{M}_X}$ are varieties and $U$ is normal, the restriction $\psi:\psi^{-1}(U)\to U$ is an open immersion.
\end{thm}
\begin{proof}
To prove $\psi$ is injective, it suffices to show that for different closed points $x$ in the special fibre $C_{\mathcal{M}_Y/\mathcal{M}_X}$, the corresponding sheaves $\mathcal{F}_x$ on $N_{Y/X}$ are not isomorphic.

The normal cone $C_{\mathcal{M}_Y/\mathcal{M}_X}$ embeds in the normal sheaf scheme: $\Spec_{\mathcal{M}_X}(\sym(J/J^2))$. By Lemma \ref{lem3.2}, we only need to check that for different points $x=(p,v)\in C_{\mathcal{M}_Y/\mathcal{M}_X},$ the Higgs pairs $(E_p,\phi_v)$ will be different.

We know from Proposition \ref{prop3.1} that $$
\psi_p:\Hom_{\mathbb{C}}(J/J^2|_p,\mathbb{C})\To \Hom_{Y}(E_p,E_p\otimes N_{Y/X}),
$$ sending $v$ to the Higgs field $\phi_v$ in (\ref{eq4}), is injective. This proves that different $x=(p,v)$ correspond to different Higgs pairs and $\psi$ is injective.

The map $\psi$ is an isomorphism on the open subsets over $\mathbb{C}^*\subset\mathbb{C}$. If $U,\psi^{-1}(U)$ are varieties and $U$ is normal, the quasi-finite birational map $\psi:\psi^{-1}(U)\to U$ is an open immersion by Zariski's main theorem.  
\end{proof}

\section{Relations to holomorphic symplectic geometry}\label{sec4}
In this section, we focus on the case of a smooth curve $C$ inside a symplectic surface $S.$ We study further properties of the moduli space $\mathcal{M}_{D_{C/S}/\mathbb{C}}$ parametrizing one-dimensional sheaves numerically equivalent to $i_*E_0$, the push-forward of a rank $r$ degree $d$ stable bundle $E_0$ on $C.$

Our first result says:
\begin{prop}\label{prop4.1}
The smooth relative moduli space $\mathcal{M}_{D_{C/S}/\mathbb{C}}$ is a Poisson variety. Its symplectic leaves are the fibres of $  \mathcal{M}_{D_{C/S}/\mathbb{C}}\to \mathbb{C},$ i.e. the Beauville--Mukai system $\mathcal{M}_{S}$ and the Hitchin system $\mathcal{M}_{\emph{GL(r)}}$.    
\end{prop}
\begin{proof}
The holomorphic symplectic form $\omega$ on $\mathcal{M}_{S}$ induces a Poisson bracket on the structure sheaf:
$$
\{\ \ ,\ \}_{\omega}:\mathcal{O}_{\mathcal{M}_{S}}\otimes_{\mathbb{C}}\mathcal{O}_{\mathcal{M}_{S}}\To \mathcal{O}_{\mathcal{M}_{S}},\ \ f\otimes g \longmapsto \{f,g\}_{\omega}.
$$

$J$ is the ideal sheaf of the Lagrangian submanifold $\mathcal{M}_{C}$, we know by \cite[Proposition 1.5.1]{CG} \begin{equation} \label{eq6}
\{J,J\}_{\omega}\subset J.
\end{equation}

Since $D_{\mathcal{M}_Y/\mathcal{M}_X} =\Spec_{\mathcal{M}_{X}}(\oplus_{n\in \mathbb{Z}} t^n J^{-n}),$ local functions on $D_{\mathcal{M}_Y/\mathcal{M}_X}$ can be uniquely written as $\sum_{i\in \mathbb{Z}} t^{-i}f_i$ with $f_i\subset J^{i}.$ The bracket $$
\{t^{-i}f_i, t^{-j}g_j\}:=t^{-i-j+1}\{f_i,g_j\}_{\omega}
$$ is well-defined because of (\ref{eq6}), and gives a Poisson bracket on $D_{\mathcal{M}_Y/\mathcal{M}_X}.$ It extends to a Poisson structure on $\mathcal{M}_{D_{C/S}/\mathbb{C}}$ since the complement of $D_{\mathcal{M}_Y/\mathcal{M}_X}$ has codimension at least 2.

This Poisson bivector lives in the vertical/fibre direction. The holomorphic symplectic varieties $\mathcal{M}_{S}$ and $\mathcal{M}_{N_{C/S}}\cong \mathcal{M}_{\text{GL(r)}}$ are exactly the symplectic leaves of $\mathcal{M}_{D_{C/S}/\mathbb{C}}$.
\end{proof}
Turning attention to abelian surfaces $A$, by Grothendieck-Riemann-Roch, 
\begin{align*}
ch(i_*E_0)&=i_*\bigl(ch(E_0)\cup\td(T_C)\bigr)\\
&=r[C]+i_*\biggl(c_1(E_0)+\frac{rc_1(T_C)}{2}\biggr).
\end{align*}

The moduli space of stable one-dimensional sheaves $\mathcal{M}_{A}$ admits the Albanese morphism
$$
\text{Alb}:\mathcal{M}_{A}\To A^{\vee} \times A,
$$ sending any stable sheaf $G\in \mathcal{M}_{A}$ to $$\bigl(\det(G)\det(i_*E_0)^{-1},ch_2(G)-ch_2(i_*E_0)\bigr)\in \Pic^{0}(A)\times \text{CH}_0(A)_0\cong A^{\vee}\times A.$$ Here $\text{CH}_0(A)_0$ is the Chow group of 0-cycles on $A$ with degree $0.$ The Albanese morphism is an isotrivial fibration and the fibre $\ku=\text{Alb}^{-1}(0,0)$ is a hyperkähler manifold of generalized Kummer type.

When restricted to $\mathcal{M}_{C}\subset \mathcal{M}_{A},$ the Albanese morphism becomes:
$$
\mathcal{M}_{C}\To A^{\vee}\times A:
\ \
\mathcal{E}\longmapsto \bigl(0,i_*c_1(\mathcal{E})-i_*c_1(E_0)\bigr),
$$
which induces a map $$\tilde{\textbf{c}}_{1}:\mathcal{M}_C\To A, 
\mathcal{E}\longmapsto i_*c_1(\mathcal{E})-i_*c_1(E_0).
$$

\begin{prop} \label{prop4.2}
The subvariety $\ku \subset \mathcal{M}_{A}$ degenerates to a smooth closed subvariety $\mathcal{N}$ in the moduli space $\mathcal{M}_{\emph{GL(r)}}$ of $\emph{GL(r)}$ stable Higgs bundles.      
\end{prop}
\begin{proof}
We first study the open subset $D_{\mathcal{M}_C/\mathcal{M}_A}\subset \mathcal{M}_{D_{C/A}/\mathbb{C}}$.
Consider the second component: $\mathcal{M}_{A}\to A$ of the Albanese morphism. The composition
$
D_{\mathcal{M}_C/\mathcal{M}_A}\to \mathcal{M}_{A}\to A
$ restricts to $$
N_{\mathcal{M}_C/\mathcal{M}_A}\To \mathcal{M}_C\xrightarrow[]{\tilde{\textbf{c}}_{1}}A
$$ on the special fibre $N_{\mathcal{M}_C/\mathcal{M}_A}\cong T^*\mathcal{M}_C\subset \mathcal{M}_{\text{GL(r)}}$.

On the other hand, the (normalized) determinant morphism $
\mathcal{M}_{A}\to A^{\vee}
$
sends $\mathcal{M}_{C}$ to $0\in A^{\vee},$ and degenerates to its derivative on the special fibre:
\begin{equation} \label{eq7}
N_{\mathcal{M}_C/\mathcal{M}_A}\To T_{[0]}(A^{\vee})\cong H^1(A,\mathcal{O}_A).
\end{equation}

Consider the standard short exact sequence
$$
0\to \ext^1_{C}(\mathcal{E},\mathcal{E})\To \ext_A^1(i_*\mathcal{E},i_*\mathcal{E}) \To \Hom_{C}(\mathcal{E},\mathcal{E}\otimes T^*C)\To 0.
$$
The trace morphism $$
\ext_A^1(i_*\mathcal{E},i_*\mathcal{E})\To H^1(A,\mathcal{O}_A)
$$ sends the subspace $\ext^1_{C}(\mathcal{E},\mathcal{E})$ to $0$ and induces 
\begin{equation} \label{eq8}
\mu:\Hom_{C}(\mathcal{E},\mathcal{E}\otimes T^*C)=\ext_A^1(i_*\mathcal{E},i_*\mathcal{E})/\ext^1_{C}(\mathcal{E},\mathcal{E})\To H^1(A,\mathcal{O}_A).
\end{equation}
This is exactly the derivative in (\ref{eq7}) at $\mathcal{E}\in \mathcal{M}_C,$ since the trace morphism is the derivative of the determinant morphism. Here we identify the space $\Hom_{C}(\mathcal{E},\mathcal{E}\otimes T^*C)$ with the space of Higgs fields.

The map $\mu$ extends uniquely to a map: $\mathcal{M}_{\text{GL(r)}}\to H^1(A,\mathcal{O}_A)$, since the complement of $T^*\mathcal{M}_C\subset \mathcal{M}_{\text{GL(r)}}$ has codimension at least $2$.
On the moduli of stable Higgs bundles, $(\tilde{\textbf{c}}_{1}, \mu): \mathcal{M}_{\text{GL(r)}}\to A\times H^1(A,\mathcal{O}_A)$ is an isotrivial fibration.

Combining the above, we see $\ku$ degenerates to the smooth subvariety $\mathcal{N}:=(\tilde{\textbf{c}}_1, \mu)^{-1}(0,0)\subset \mathcal{M}_{\text{GL(r)}}.$ 
\end{proof}
We now give a more explicit description of $\mathcal{N}.$ The short exact sequence $$
0\to \ext^1_{C}(\mathcal{E},\mathcal{E})\To \ext_A^1(i_*\mathcal{E},i_*\mathcal{E}) \To \Hom_{C}(\mathcal{E},\mathcal{E}\otimes T^*C)\To 0
$$ is self dual with respect to the Serre pairings on $C$ and $A$. The (Serre) dual of the trace map is the identity map $$
H^1(A,\mathcal{O}_A)\To \ext_A^1(i_*\mathcal{E},i_*\mathcal{E}),
$$ and the dual of $
\mu:\Hom_{C}(\mathcal{E},\mathcal{E}\otimes T^*C)\to H^1(A,\mathcal{O}_A)
$ in (\ref{eq8}) is 
\begin{equation} \label{eq9}
id:H^1(A,\mathcal{O}_A)\To \ext^1_{C}(\mathcal{E},\mathcal{E})\subset \ext_A^1(i_*\mathcal{E},i_*\mathcal{E}).
\end{equation}
Therefore, the kernel of $\mu$ is exactly the Higgs fields $\phi\in \Hom_{C}(\mathcal{E},\mathcal{E}\otimes T^*C)$ satisfying $$
\phi \cup id(H^1(A,\mathcal{O}_A))=0\in \ext^1_{C}(\mathcal{E},\mathcal{E}\otimes T^*_C)
$$ and we see:
\begin{prop}
The closed subvariety $\mathcal{N}$ is the locus of Higgs bundles $(\mathcal{E},\phi)$ such that \begin{equation} \label{eq10}
i_*c_1(\mathcal{E})=i_*c_1(E_0)\in A,\text{  and } \phi \cup id(H^1(A,\mathcal{O}_A))=0.
\end{equation}
\end{prop}
\begin{rmk}
When $C\xhookrightarrow{i}A$ is the Abel-Jacobi map of a genus $2$ curve, $A\cong\Pic^0(A)\cong \Pic^0(C)$. It is clear from \emph{(\ref{eq10})} that $\mathcal{N}$ is the moduli of \emph{SL(r)} Higgs bundles on $C.$
\end{rmk}
We now relate $\mathcal{N}$ to a holomorphic symplectic reduction $\mathcal{N}^{\vee}$ of $\mathcal{M}_{\text{GL(r)}}$. For any line bundle $L\in\Pic^{0}(A)$, the map that sends a Higgs bundle $(\mathcal{E},\phi)$ to $(\mathcal{E}\otimes i^*L,\phi)$ defines a $\Pic^{0}(A)$ action on $\mathcal{M}_{\text{GL(r)}},$ where $i^*$ is the pull-back map: $\Pic^0(A)\to \Pic^0(C)$.   

\begin{prop} \label{prop4.4}
The $\Pic^{0}(A)$ action is Hamiltonian and the corresponding holomorphic symplectic reduction $\mathcal{N}^{\vee}$ is a dual Lagrangian fibration of the symplectic subvariety $\mathcal{N}$ in $\mathcal{M}_{\emph{GL(r)}}$. Moreover, there exists a finite map $\mathcal{N}\to \mathcal{N}^{\vee}.$  
\end{prop}
\begin{proof}
On the open subset $T^*\mathcal{M}_C\subset\mathcal{M}_{\text{GL(r)}}$, the holomorphic symplectic form is the usual exterior derivative of the tautological 1-form. The $\Pic^{0}(A)$ action on $T^*\mathcal{M}_C\subset\mathcal{M}_{\text{GL(r)}}$ is induced by the $\Pic^{0}(A)$ action on $\mathcal{M}_C$, which sends $\mathcal{E}$ to $\mathcal{E}\otimes i^*L$.

For Lie algebra elements $\xi \in H^1(\mathcal{O}_A)=T_{[0]}\Pic^0(A)$, denote the corresponding infinitesimal action (equivalently, the induced fundamental vector fields) on $\mathcal{M}_C$ by $X_{\xi}$. At a point $(\mathcal{E},v)\in T^*\mathcal{M}_C$ where $v$ is a cotangent vector, the following standard pairing  defines a moment map 
\begin{equation}\label{eq11}
\mu_0=<v,X_\xi>.
\end{equation}

We claim that the moment map $\mu_0$ on $T^*\mathcal{M}_C$ is the morphism $\mu$ in (\ref{eq8}), where the source of $\mu$ is interpreted as the space of Higgs fields. Indeed, the (Serre) dual of $\mu$ is the identity map in (\ref{eq9}):
$$
H^1(A,\mathcal{O}_A)\xrightarrow[]{id} \ext^1_{C}(\mathcal{E},\mathcal{E}),
$$ which is exactly the $\Pic^{0}(A)$ infinitesimal action on $\mathcal{M}_C$. We then know $\mu$ is also the moment map on $\mathcal{M}_{\text{GL(r)}}$ by continuity.

The holomorphic symplectic reduction $\mathcal{N}^{\vee}$ is the orbifold $\mu^{-1}(0)/\Pic^0(A).$ The closed subvariety $\mathcal{N}=(\tilde{\textbf{c}}_1, \mu)^{-1}(0,0)$ is a slice of the $\Pic^0(A)$ action on $\mu^{-1}(0)$ and the map $\mathcal{N}\to \mu^{-1}(0)/\Pic^0(A)=\mathcal{N}^{\vee}$ is finite. This shows  $\mathcal{N}$ is a symplectic subvariety of $\mathcal{M}_{\text{GL(r)}}$.

It remains to prove $\mathcal{N}$ and $\mathcal{N}^{\vee}$ are dual Lagrangian fibrations. Since the symplectic structures on $\mathcal{N}$ and $\mathcal{N}^{\vee}$ are induced by the symplectic structure on $\mathcal{M}_{\text{GL(r)}}$, both spaces admit Lagrangian fibrations to the linear subspace of the Hitchin base $\oplus_{i=1}^{r}H^0(\Omega_C^i)$ defined by the equation $$
\text{tr}(\phi)\cup i^*H^1(A,\mathcal{O}_A)=0\in H^1(C,\Omega_C), 
$$ where $i^*$ is the restriction map: $H^1(A,\mathcal{O}_A)\to H^1(C,\mathcal{O}_C)$. The generic element of this subspace represents a smooth spectral curve $C_0$ in $T^*C$.

The spectral curve $C_0$ maps to $A$ via the composition $$
f: C_0\subset T^*C\To C \xhookrightarrow[]{i} A.
$$ The fibre of $\mathcal{N}$ over $[C_0]$ is exactly the kernel of the corresponding Albanese morphism $$
\Pic^0(C_0)\To A,
$$ and the fibre of $\mathcal{N}^{\vee}$ over $[C_0]$ is the cokernel of the pull-back map $$
\Pic^0(A)=A^{\vee}\xrightarrow[]{f^*} \Pic^0(C_0).
$$ It then follows that $\mathcal{N}$ and $\mathcal{N}^{\vee}$ are dual fibrations since the two morphisms above are dual to each other.
\end{proof}

\begin{rmk}
In \emph{\cite[Appendix]{res}}, the authors ask whether there exist compact symplectic varieties that degenerate to the moduli of \emph{SL(r)} Higgs bundles on curves with genus $g>2$, generalizing the genus $2$ example. We have provided answers for the opposite direction here.
\end{rmk}

\section{Reduced Atiyah classes} \label{sec5}
We prove Proposition \ref{prop3.1} in this section. We start with the following standard lemma:
\begin{lemma} \label{lem5.1}
Let $\mathcal{G}$ be a complex of quasi-coherent sheaves on a quasi-projective scheme $M$ and $\mathcal{H}^i(\mathcal{G})=0$ for $i>0.$  For a quasi-coherent sheaf $H$ on $M,$ there exists a natural isomorphism 
\begin{equation}\label{eq12}
\mathcal{H}^0(R\mathcal{H}om(\mathcal{G},H))\cong \mathcal{H}om(\mathcal{H}^0(\mathcal{G}),H).
\end{equation}

Similarly, suppose $f:N\to M$ is a proper morphism of quasi-projective schemes, then there exist natural isomorphisms
\begin{equation} \label{eq13}
\mathcal{H}^0(Lf^*\mathcal{G})\cong f^*\mathcal{H}^0(\mathcal{G}), \text{ and } \mathcal{H}^0(Rf_*\mathcal{K})\cong f_*\mathcal{H}^0(\mathcal{K})
\end{equation}
for any complex $\mathcal{K}$ of quasi-coherent sheaves on $N$ such that $\mathcal{H}^{i}(\mathcal{K})=0$ for $i<0.$
\end{lemma}
\begin{proof}
We prove the isomorphism in (\ref{eq12}) and the isomorphisms in (\ref{eq13}) are analogous. Consider the spectral sequence in \cite[(3.8)]{fm}: 
\begin{equation} \label{exts}
E_2^{p,q}=\mathcal{E}xt^p(\mathcal{H}^{-q}(\mathcal{G}),H)\Rightarrow\mathcal{E}xt^{p+q}(\mathcal{G},H).
\end{equation}
When $p<0,$ $E_2^{p,q}=\mathcal{E}xt^p(\mathcal{H}^{-q}(\mathcal{G}),H)=0.$ Since $\mathcal{H}^i(\mathcal{G})=0$ if $i>0$, $E_2^{p,q}=0$  when $q<0.$ The differentials of the spectral sequence
$$
d_r^{p,q}:E_r^{p,q}\To E_r^{p+r,q-r+1}
$$ also vanish if either $(p,q)$ or $(p+r,q-r+1)$ equals (0,0), which gives 
\[
\mathcal{H}om(\mathcal{G},H))=E_{\infty}^{0,0}\cong \mathcal{H}om(\mathcal{H}^0(\mathcal{G}),H). 
\qedhere
\]
\end{proof}

We now recall the concept of the Atiyah class as explained in \cite[Section 1.4]{At}. The Atiyah classes of the universal sheaves $F$ and $E$, relative to the bases $X$ and $Y$ respectively, are morphisms
\begin{equation*}
\At_F:   R\pi_{\mathcal{M}_X*}(F^{\vee}\otimes^{\textbf{L}}F)^{\vee} [-1]\To \mathbb{L}_{\mathcal{M}_{X}},\text{   and } \At_E: R\pi_{\mathcal{M}_Y*}(E^{\vee}\otimes^{\textbf{L}}E)^{\vee} [-1]\To \mathbb{L}_{\mathcal{M}_{Y}},
\end{equation*} where $\mathbb{L}$ denotes the cotangent complexes.

Consider the exact triangle associated with the adjunction $Li^*i_*E\xrightarrow{\iota}{\text{}} E$:
\begin{equation} \label{eq15}
Li^*i_*E\xrightarrow{\iota}{\text{}} E \to \text{Cone}(\iota).
\end{equation}
By the functoriality of the Atiyah class, we have the commutative diagram 
\begin{equation} \label{eq16}
\begin{tikzcd}
R\pi_{\mathcal{M}_Y*}((Li^*i_*E)^{\vee}\otimes^{\textbf{L}}E)^{\vee}[-1] \arrow{r}{\iota_*} \arrow{d}{Lj^*\At_F} & R\pi_{\mathcal{M}_Y*}(E^{\vee}\otimes^{\textbf{L}}E)^{\vee}[-1] \arrow{r}{}  \arrow{d}{\At_E}\textbf{}& \text{Cone}(\iota_*)\\ Lj^*\mathbb{L}_{\mathcal{M}_X} \arrow{r}{} &  \mathbb{L}_{\mathcal{M}_{Y}} \arrow{r}{} &\mathbb{L}_{\mathcal{M}_{Y}/\mathcal{M}_{X}},  
\end{tikzcd}
\end{equation}
where the first row is obtained by applying $ R\pi_{\mathcal{M}_Y*}((\ \ \ )^{\vee}  \otimes^{\textbf{L}}E)^{\vee}[-1] $ to the exact triangle (\ref{eq15}). Here we have used the natural isomorphisms: $$
R\pi_{\mathcal{M}_Y*}((Li^*i_*E)^{\vee}\otimes^{\textbf{L}}E)^{\vee}\cong R\pi_{\mathcal{M}_Y*}((i_*E)^{\vee}\otimes^{\textbf{L}}i_*E)^{\vee}\cong R\pi_{\mathcal{M}_Y*}((j^*F)^{\vee}\otimes^{\textbf{L}}j^*F)^{\vee},
$$
and the base change theorem \cite[\href{https://stacks.math.columbia.edu/tag/08IB}{Tag 08IB}]{stacks-project} for the flat map $X\times\mathcal{M}_X\xrightarrow[]{\pi_{\mathcal{M}_X}} \mathcal{M}_X
$:
$$
R\pi_{\mathcal{M}_Y*}((j^*F)^{\vee}\otimes^{\textbf{L}}j^*F)^{\vee}\cong R\pi_{\mathcal{M}_Y*}(Lj^*(F^{\vee}\otimes^{\textbf{L}}F)^{\vee})\cong Lj^*(R\pi_{\mathcal{M}_X*}(F^{\vee}\otimes^{\textbf{L}}F)^{\vee})
$$
to obtain the left vertical arrow $Lj^*\At_F$.

Denote the inclusions of the closed point $p\subset \mathcal{M}_Y,\mathcal{M}_X$ by $i_p.$ The point $p$ corresponds to the stable sheaf $E_p$ on $Y$.
We restrict the commutative diagram (\ref{eq16}) to this closed point and then dualize, using again the base change theorem \cite[\href{https://stacks.math.columbia.edu/tag/08IB}{Tag 08IB}]{stacks-project} we obtain: 
\begin{equation} \label{eq17}
\begin{tikzcd}
(Li_p^*\mathbb{L}_{\mathcal{M}_{Y}})^{\vee}\arrow{r} \arrow{d}{(Li_p^*\At_E)^{\vee}} &(Li_p^*\mathbb{L}_{\mathcal{M}_{X}})^{\vee}\arrow{r}{} \arrow{d}{(Li_p^*\At_F)^{\vee}} &(Li_p^*\mathbb{L}_{\mathcal{M}_Y/\mathcal{M}_{X}})^{\vee}[1] \arrow[d,dotted] \\
R\pi_{p*}R\mathcal{H}om_{Y}(E_p,E_p)[1] 
\arrow{r}{\iota_{p}^*} & R\pi_{p*}R\mathcal{H}om_{X}(i_*E_p,i_*E_p)[1] \arrow{r}{} & \text{Cone}(\iota_{p}^{*}).
\end{tikzcd}
\end{equation}
Here the restriction $\iota_p$ of $\iota$ is the adjunction $$
Li^*i_{*}E_p\xrightarrow{\iota_p}{\text{}} E_p.
$$

Since both rows of (\ref{eq17}) are exact triangles, there exists an arrow $$
(Li_p^*\mathbb{L}_{\mathcal{M}_Y/\mathcal{M}_{X}})^{\vee}[1] \To
\text{Cone}(\iota_{p}^{*})= R\pi_{p*}(\text{Cone}(\iota_p)^{\vee}\otimes^{\textbf{L}}E_p)[2],$$ which makes the diagram commute. Its induced maps on the cohomology groups are unique though the map itself may not be. Taking the degree $0$ cohomology, we get what we call the reduced Atiyah class 
\begin{equation} \label{eq18}
\Atp:\mathcal{H}^{1}((Li_p^*\mathbb{L}_{\mathcal{M}_Y/\mathcal{M}_{X}})^{\vee}) \To
\pi_{p*}(\text{Cone}(\iota_p)^{\vee}\otimes^{\textbf{L}}E_p[2]).
\end{equation}

\begin{prop}\label{prop5.1}
The reduced Atiyah class  $\Atp$ is a morphism: $$
\Hom_{\mathbb{C}}(J/J^2|_p,\mathbb{C})\To \Hom_{Y}(E_p,E_p\otimes N_{Y/X}),
$$ and is injective.  
\end{prop}
\begin{proof}
Consider the following part of the long exact sequences associated with (\ref{eq17}):
\begin{equation*} 
\begin{tikzcd}[column sep=small]
\mathcal{H}^0((Li_p^*\mathbb{L}_{\mathcal{M}_{Y}})^{\vee})\arrow{r} \arrow[d] &\mathcal{H}^0((Li_p^*\mathbb{L}_{\mathcal{M}_{X}})^{\vee})\arrow{r}{} \arrow{d} &\mathcal{H}^1((Li_p^*\mathbb{L}_{\mathcal{M}_Y/\mathcal{M}_{X}})^{\vee}) \arrow{d}{\Atp}\arrow{r} & \mathcal{H}^1((Li_p^*\mathbb{L}_{\mathcal{M}_{Y}})^{\vee}) \arrow{d}\\
\ext^1_{Y}(E_p,E_p) 
\arrow{r}{\iota_{p}^*} & \ext^1_{X}(i_*E_p,i_*E_p) \arrow{r}{} & \pi_{p*}(\text{Cone}(\iota_p)^{\vee}\otimes^{\textbf{L}}E_p[2]) \arrow{r}& \ext^2_{Y}(E_p,E_p).
\end{tikzcd}
\end{equation*}
As we are on the closed point $p$, all terms in the diagram are vector spaces supported on $p$.

Apply Lemma \ref{lem5.1} to the functors $(\ \ )^{\vee}$ and $Li_p^*,$ we see the term
$$
\mathcal{H}^0((Li_p^*\mathbb{L}_{\mathcal{M}_{Y}})^{\vee})\cong (i_p^*\Omega_{\mathcal{M}_Y})^*
$$ is the tangent space at $p$ of $\mathcal{M}_Y.$ Similarly, $\mathcal{H}^0((Li_p^*\mathbb{L}_{\mathcal{M}_{X}})^{\vee})$ is the tangent space of $\mathcal{M}_X,$ and $\mathcal{H}^1((Li_p^*\mathbb{L}_{\mathcal{M}_Y/\mathcal{M}_{X}})^{\vee})$ is the normal vector space $\Hom_{\mathbb{C}}(J/J^2|_p,\mathbb{C}).$ 

As explained in \cite[p.248, 250]{fm}, the cohomology groups $\mathcal{H}^{-i}(Li^*i_*E_p)$ are $\wedge^{i}N^*_{Y/X}\otimes E_p$. Consequently the perfect complex $\text{Cone}(\iota_p)$ is concentrated in degrees $\leq-2$, and we obtain $$
\pi_{p*}(\text{Cone}(\iota_p)^{\vee}\otimes^{\textbf{L}}E_p[2])\cong\Hom_{Y}(\text{Cone}(\iota_p)[-2],E_p)\cong\Hom_{Y}(E_p,E_p\otimes N_{Y/X})
$$ by applying Lemma \ref{lem5.1} to the functors $R\mathcal{H}om_{Y}(\ \ \ , E_p)$ and $R\pi_{p*}.$ This proves that $\Atp$ is a morphism: $$
\Hom_{\mathbb{C}}(J/J^2|_p,\mathbb{C})\To \Hom_{Y}(E_p,E_p\otimes N_{Y/X}).
$$

Since the Atiyah classes $\At_F,\At_E$ provide obstruction theories of $\mathcal{M}_{X}$ and $\mathcal{M}_{Y}$ respectively by \cite[Theorem 4.1]{HT}, we know the vertical maps in the above commutative diagram: $$
(i_p^*\Omega_{\mathcal{M}_Y})^*\To \ext^1_{Y}(E_p,E_p), \text{and } (i_p^*\Omega_{\mathcal{M}_X})^*\To \ext^1_{X}(i_*E_p,i_*E_p)
$$ are isomorphisms by \cite[Theorem 4.5]{bf} and the vertical map from the universal obstruction space: $$
\mathcal{H}^1((Li_p^*\mathbb{L}_{\mathcal{M}_{Y}})^{\vee})\To \ext^2_{Y}(E_p,E_p)
$$ is injective by \cite[Lemma 4.6]{bf}. The five lemma then gives the injectivity of $\Atp$.
\end{proof}
To prove Proposition \ref{prop3.1} and show that $\psi_p$ in (\ref{eq5}) is injective, it suffices to prove the following result.

\begin{prop} \label{prop5.2}
The reduced Atiyah class
$$
\Atp:\Hom_{\mathbb{C}}(J/J^2|_p,\mathbb{C})\To \Hom_{Y}(E_p,E_p\otimes N_{Y/X})
$$  
is equal to $\psi_p$ in \emph{(\ref{eq5})}.   
\end{prop}
We use the deformation to the normal cone family $D_{Y/X}\to \mathbb{C}$ to reduce Proposition \ref{prop5.2} to the linear case: i.e. when $X$ is a vector bundle over the zero section $Y,$ which is now quasi-projective. 
\begin{lemma} \label{lem5.2}
Proposition \emph{\ref{prop5.2}} holds in the case that $X=\Tot_Y(V)$ is a vector bundle $V$ over the zero section $Y$.    
\end{lemma}
\begin{proof}
Consider, as in the proof of Proposition \ref{prop5.1}, the following commutative diagram associated with (\ref{eq17}):
\begin{equation}\label{eq19}
\begin{tikzcd}
T_p\mathcal{M}_{X}\arrow[r,->>] \arrow{d}{\simeq}&   \Hom_{\mathbb{C}}(J/J^2|_p,\mathbb{C}) \arrow{d}{\Atp}\\
\ext^1_{X}(i_*E_p,i_*E_p) \arrow{r}{} & \Hom_{Y}(E_p,E_p\otimes N_{Y/X}).
\end{tikzcd}    
\end{equation}
Here $T_p\mathcal{M}_{X}$ is the tangent space of $\mathcal{M}_X$ at $p.$
Since $X=\Tot_Y(V)$, pushing forward sheaves via the projection $$X=\Tot_Y(V)\to Y$$ induces a map from an open neighbourhood of $\mathcal{M}_Y\subset\mathcal{M}_{X}$ to $\mathcal{M}_Y$, which restricts to the identity map on $\mathcal{M}_Y.$ The tangent space $T_p\mathcal{M}_{X}$ splits as the direct sum of $T_p\mathcal{M}_{Y}$ and the normal vector space $\Hom_{\mathbb{C}}(J/J^2|_p,\mathbb{C})$. The upper horizontal map is therefore surjective and the diagram (\ref{eq19}) completely determines the map $\Atp$.

To show $\Atp=\psi_p$, we then only need to prove the following diagram commutes: 
\begin{equation}\label{eq20}
\begin{tikzcd}
T_p\mathcal{M}_{X}\arrow[r,->>] \arrow{d}{\simeq}&   \Hom_{\mathbb{C}}(J/J^2|_p,\mathbb{C}) \arrow{d}{\psi_p}\\
\ext^1_{X}(i_*E_p,i_*E_p) \arrow{r}{} & \Hom_{Y}(E_p,E_p\otimes N_{Y/X}).
\end{tikzcd}    
\end{equation}

Pick any class $\alpha\in \ext^1_{X}(i_*E_p,i_*E_p),$ which corresponds to an extension sequence on $X$:
\begin{equation}\label{eq21}
0\To i_*E_p\To \mathbf{E}\xrightarrow[]{r} i_*E_p \to 0.  
\end{equation}
The extension class $\alpha$ is a morphism in the derived category of $Y$:
$$
Li^*i_*E_p\xrightarrow[]{\alpha}E_p[1].
$$ Taking its degree $-1$ cohomology gives $$
\mathcal{H}^{-1}(\alpha):\mathcal{H}^{-1}(Li^*i_*E_p)\cong N^*_{Y/X}\otimes E_p \To E_p,
$$ which is the lower horizontal arrow in the diagram (\ref{eq20}). Here the isomorphism  $\mathcal{H}^{-1}(Li^*i_*E_p)\cong N^*_{Y/X}\otimes E_p$ is obtained in \cite[p.248, 250]{fm} via the Koszul resolution.

On the other hand, denote by $\pi$ the projection $X=\Tot_Y(V)\to Y$ and 
consider the multiplication by $\pi^*V^*=\pi^*N^*_{Y/X}$ on the first order deformation $\mathbf{E}$ of $i_*E_p$ in (\ref{eq21}). The images of $$\pi^*V^*\otimes \mathbf{E}\To \mathbf{E}\xrightarrow[]{r} i_*E_p, \text{ and } \pi^*V^*\otimes i_*E_p\To \pi^*V^*\otimes \mathbf{E}\To \mathbf{E}$$ are zero as the multiplication $\pi^*V^*\cdot i_*E=0$. We therefore get a morphism 
\begin{equation*} 
\pi^*V^*\otimes\mathbf{E}/(\pi^*V^*\otimes i_*E_p) \cong \pi^*V^*\otimes i_*E_p\To i_*E_p,
\end{equation*}
whose push-forward via $\pi_*$: 
\begin{equation}\label{eq22}
 N_{Y/X}^*\otimes E_p\To E_p
\end{equation}
is the image of the class $\alpha$
via the composition in the diagram (\ref{eq20}):
$$
\ext^1_{X}(i_*E_p,i_*E_p)\cong T_p\mathcal{M}_X\To \Hom_{\mathbb{C}}(J/J^2|_p,\mathbb{C}) \xrightarrow[]{\psi_p} \Hom_{Y}(E_p,E_p\otimes N_{Y/X})
$$
since $\psi_p$ in (\ref{eq5}) is defined in the same way as above.

It remains to prove the two morphisms associated with $\alpha$: $\mathcal{H}^{-1}(\alpha)$ and (\ref{eq22})  are equal. By the local nature of the two maps, we are free to restrict to small open subsets of $X.$

Since $Y$ is the zero section of $X,$ $\mathcal{O}_{Y}$ on $X$ is resolved by the Koszul complex $$V^{\bullet}:=\cdots\To\wedge^2\pi^*V^{*}\To \pi^*V^{*} \To \mathcal{O}_{X},$$ where the differentials are given by contractions with the tautological section of $\pi^*V$ on $X$. Choose a locally free resolution $L^{\bullet}=(\dots \to L^{-1}\to L^0)$ of the sheaf $E_p$ on $Y$, $V^{\bullet}\otimes \pi^*L^{\bullet}$ is a locally free resolution of $i_*E_p$ on $X$.

Shrinking $X$ if necessary, $\alpha\in \ext^1_{X}(i_*E_p,i_*E_p)$ is represented by a morphism of complexes $$
f: V^{\bullet}\otimes \pi^*L^{\bullet}\To i_*E_p[1],
$$ and in particular we get a morphism in degree $-1$ 
$$
f_{-1}:\pi^*(L^0\otimes V^{*})\oplus\pi^*L^{-1}\To i_*E_p. 
$$ 
The adjunction of $f_{-1}$ is the map: $
L^0\otimes V^*\oplus L^{-1} \to E_p,
$ which induces $$
\mathcal{H}^{-1}(\alpha): (L^0/L^{-1})\otimes V^*\cong E_p\otimes V^* \To E_p.
$$

The middle term  $\mathbf{E}$ in the extension sequence (\ref{eq21}) is the mapping cone of $$
f[-1]:V^{\bullet}\otimes \pi^*L^{\bullet}[-1]\To i_*E_p,
$$ and the sequence (\ref{eq21}) is the same as 
\begin{equation*}
0\To i_*E_p \To(\pi^*L^{0}\oplus i_*E_p)/(\pi^*(L^0\otimes V^{*})\oplus\pi^*L^{-1})\To \pi^*L^{0}/(\pi^*(L^0\otimes V^{*})\oplus\pi^*L^{-1})\To 0.
\end{equation*}

It follows by direct calculations, using the above extension sequence, that $\mathcal{H}^{-1}(\alpha)$ equals the map in (\ref{eq22}).
\end{proof}
We now deal with the general case:
\begin{proof}[Proof of Proposition \emph{\ref{prop5.2}}]
The family $D_{Y/X}\to \mathbb{C}$ and the Atiyah class of the sheaf $\mathcal{F}$ (constructed in Proposition \ref{prop2.1}) on $D_{Y/X}\times_{\mathbb{C}}D_{\mathcal{M}_Y/\mathcal{M}_X}$  relate the general case to the linear case proved above.

The sheaf $\mathcal{F}$ is flat over $D_{\mathcal{M}_Y/\mathcal{M}_X}$ and the projection $$\pi^{}_{D_{\mathcal{M}_{Y}/\mathcal{M}_{X}}}:D_{Y/X}\times_{\mathbb{C}}D_{\mathcal{M}_Y/\mathcal{M}_X}\To D_{\mathcal{M}_Y/\mathcal{M}_X}$$ is smooth. By Hilbert's syzygy theorem, $\mathcal{F}$ is perfect on $D_{Y/X}\times_{\mathbb{C}}D_{\mathcal{M}_Y/\mathcal{M}_X}$.

We have the following Atiyah class  relative to the base $D_{Y/X}$: $$
\mathcal{F}[-1]\To \mathcal{F}\otimes^{\mathbf{L}}\mathbb{L}_{\pi^{}_{D_{Y/X}}}\cong \mathcal{F}\otimes^{\mathbf{L}}L\pi^{*}_{D_{\mathcal{M}_{Y}/\mathcal{M}_{X}}}\mathbb{L}_{D_{\mathcal{M}_{Y}/\mathcal{M}_{X}}\big/\mathbb{C}},
$$
where the last isomorphism follows from the flatness of 
$D_{Y/X}\to \mathbb{C}$. This induces $$
\At_{\mathcal{F}}:R\pi^{}_{D_{\mathcal{M}_{Y}/\mathcal{M}_{X}}*}(\mathcal{F}^{\vee}\otimes^{\mathbf{L}}\mathcal{F})^{\vee}[-1]\To \mathbb{L}_{D_{\mathcal{M}_{Y}/\mathcal{M}_{X}}\big/\mathbb{C}}.
$$

Consider the closed immersions $\mathbf{i}:Y\times\mathbb{C} \xhookrightarrow{}D_{Y/X}$ and $\mathbf{j}:\mathcal{M}_{Y}\times \mathbb{C}\xhookrightarrow{}D_{\mathcal{M}_{Y}/\mathcal{M}_{X}}.$ The conormal sheaf $t^{-1}I/t^{-2}I^2$ of $\mathbf{i}$ is isomorphic to the pull-back of $I/I^2$ via the projection $$
\pi_Y:Y\times \mathbb{C} \To Y.
$$ 
The conormal sheaf of $\mathbf{j}$ is isomorphic to $\pi_{\mathcal{M}_{Y}}^*(J/J^2)$.

Denote by $\mathbf{E}$ the (trivial) pull-back of $E$ to $Y\times \mathcal{M}_Y\times \mathbb{C}$, $\mathbf{i}_*\mathbf{E}\cong \mathbf{j}^*\mathcal{F}$. Similar to the diagram (\ref{eq16}), we have:
\begin{equation*} 
\begin{tikzcd}
L\mathbf{j}^*R\pi^{}_{D_{\mathcal{M}_{Y}/\mathcal{M}_{X}}*}(\mathcal{F}^{\vee}\otimes^{\mathbf{L}}\mathcal{F})^{\vee}[-1] \arrow{r}{} \arrow{d}{L\mathbf{j}^*\At_{\mathcal{F}}} & R\pi^{}_{\mathcal{M}_Y\times \mathbb{C}*}(\mathbf{E}^{\vee}\otimes^{\mathbf{L}}\mathbf{E})^{\vee}[-1] \arrow{r}{} \arrow{d}{\At_{\mathbf{E}}} & \text{Cone}\arrow{d}{} \\
L\mathbf{j}^*\mathbb{L}_{D^{}_{\mathcal{M}_{Y}/\mathcal{M}_{X}}\big/\mathbb{C}} \arrow{r}{} & \mathbb{L}_{\mathcal{M}_{Y}\times \mathbb{C}/\mathbb{C}} \arrow{r}{} & \mathbb{L}_{\mathcal{M}_{Y}\times \mathbb{C}\big/D^{}_{\mathcal{M}_{Y}/\mathcal{M}_{X}}}.  
\end{tikzcd}
\end{equation*}
As in (\ref{eq16}), here we use the base change theorem \cite[\href{https://stacks.math.columbia.edu/tag/08IB}{Tag 08IB}]{stacks-project} and the flatness of $$
\pi^{}_{D_{\mathcal{M}_{Y}/\mathcal{M}_{X}}}:D_{Y/X}\times_{\mathbb{C}}D_{\mathcal{M}_Y/\mathcal{M}_X}\To D_{\mathcal{M}_Y/\mathcal{M}_X}
$$
to obtain the left map in the first row: $$
L\mathbf{j}^*R\pi^{}_{D_{\mathcal{M}_{Y}/\mathcal{M}_{X}}*}(\mathcal{F}^{\vee}\otimes^{\mathbf{L}}\mathcal{F})^{\vee}
\cong
R\pi^{}_{\mathcal{M}_Y\times \mathbb{C}*}((\mathbf{j}^*\mathcal{F})^{\vee}\otimes^{\mathbf{L}}\mathbf{j}^*\mathcal{F})^{\vee}
\To R\pi^{}_{\mathcal{M}_Y\times \mathbb{C}*}(\mathbf{E}^{\vee}\otimes^{\mathbf{L}}\mathbf{E})^{\vee}.
$$

Denote the inclusion of the affine line: $p\times\mathbb{C}\xhookrightarrow[]{}\mathcal{M}_{Y}\times \mathbb{C}$ by $i_{p\times\mathbb{C}}$.
Pull the above diagram back by $i_{p\times\mathbb{C}}$ and then dualize, we obtain the morphism
$$
\mathbf{F}:\mathcal{H}^{1}((Li_{p\times\mathbb{C}}^*\mathbb{L}^{}_{\mathcal{M}_{Y}\times \mathbb{C}\big/D^{}_{\mathcal{M}_{Y}/\mathcal{M}_{X}}})^{\vee}) \To
\mathcal{H}^{1}((Li_{p\times\mathbb{C}}^*\text{Cone})^{\vee}).
$$
The same arguments as in the first half of Proposition \ref{prop5.1} show that $\mathbf{F}$ is a map between two trivial vector bundles on $p\times\mathbb{C}$ with fibres $\Hom_{\mathbb{C}}(J/J^2|_p,\mathbb{C})$ and $\Hom_{Y}(E_p,E_p\otimes N_{Y/X})$ respectively. The trivializations are obtained by our descriptions of the conormal sheaves of $\mathbf{i}$ and $\mathbf{j}$.

The restrictions of $\mathbf{F}$ then give a $\mathbb{C}$-family of morphisms 
$$\mathbf{F}_t:\Hom_{\mathbb{C}}(J/J^2|_p,\mathbb{C})\To \Hom_{Y}(E_p,E_p\otimes N_{Y/X})$$ on each fibre $p\times \{t\}$ of the family: $p\times\mathbb{C}\to\mathbb{C}$. The maps $\mathbf{F}_{t\neq 0}=\Atp$ since $\mathcal{F}$ restricts to the universal sheaf $F$ on the fibre $X\times \mathcal{M}_{X}$ over $t\neq 0$. By continuity, $\mathbf{F}_0=\mathbf{F}_{t\neq0}=\Atp.$

We show now $\mathbf{F}_0$ is equal to $\psi_p$ in (\ref{eq5}), which finishes the proof. Denote the restriction of $\mathcal{F}$ to the special fibre $C_{\mathcal{M}_Y/\mathcal{M}_X}\times N_{Y/X}$ by $\mathcal{F}_0.$ As in (\ref{eq19}), we have the commutative diagram:
\begin{equation}\label{eq23}
\begin{tikzcd}
T_pC^{}_{\mathcal{M}_Y/\mathcal{M}_X}\arrow[r,->>] \arrow{d}{\mathcal{H}^0((Li_p^*\At_{\mathcal{F}_0}{})^{\vee})}&   \Hom_{\mathbb{C}}(J/J^2|_p,\mathbb{C}) \arrow{d}{\mathbf{F}_0}\\
\ext^1_{N_{Y/X}}(i_*E_p,i_*E_p) \arrow{r}{} & \Hom_{Y}(E_p,E_p\otimes N_{Y/X}).
\end{tikzcd}    
\end{equation}
The right vertical arrow is $\mathbf{F}_0$ by naturality.
The upper horizontal map is surjective since $C^{}_{\mathcal{M}_Y/\mathcal{M}_X}$ is a cone over $\mathcal{M}_Y,$ therefore to show $\mathbf{F}_0=\psi_p$ we only need to check the commutativity of 
\begin{equation}\label{eq24}
\begin{tikzcd}
T_pC^{}_{\mathcal{M}_Y/\mathcal{M}_X}\arrow[r,->>] \arrow{d}{\mathcal{H}^0((Li_p^*\At_{\mathcal{F}_0}{})^{\vee})}&   \Hom_{\mathbb{C}}(J/J^2|_p,\mathbb{C}) \arrow{d}{\psi_p}\\
\ext^1_{N_{Y/X}}(i_*E_p,i_*E_p) \arrow{r}{} & \Hom_{Y}(E_p,E_p\otimes N_{Y/X}).
\end{tikzcd}    
\end{equation}

This follows from the functoriality of the Atiyah classes and the linear case proved in Lemma \ref{lem5.2}. Indeed, the sheaf $\mathcal{F}_0$ on $C_{\mathcal{M}_Y/\mathcal{M}_X}\times N_{Y/X}$ induces a classfying map
$$
\psi_0:C^{}_{\mathcal{M}_Y/\mathcal{M}_X}\To \mathcal{M}^{}_{N_{Y/X}}.
$$ The diagram (\ref{eq24}) factorizes as
$$
\begin{tikzcd}
T_pC^{}_{\mathcal{M}_Y/\mathcal{M}_X}\arrow[r] \arrow{d}{\psi_{0*}}&   \Hom_{\mathbb{C}}(J/J^2|_p,\mathbb{C}) \arrow{d}{}\\
T_p\mathcal{M}^{}_{N_{Y/X}} \arrow{r}{} \arrow{d}{}& N_{\mathcal{M}^{}_Y/\mathcal{M}^{}_{N_{Y/X}}}(p) \arrow{d}{}
\\
\ext^1_{N_{Y/X}}(i_*E_p,i_*E_p) \arrow{r}{} & \Hom_{Y}(E_p,E_p\otimes N_{Y/X}),    
\end{tikzcd}
$$ where $N_{\mathcal{M}^{}_Y/\mathcal{M}^{}_{N_{Y/X}}}(p)$ is the normal vector space of $\mathcal{M}_Y\xhookrightarrow{}\mathcal{M}^{}_{N_{Y/X}}$ at the closed point $p$. Its upper half is the commutative diagram associated with
$$
\xymatrix{
		\,\mathcal{M}_Y \,\ar@{^{(}->}[r]\ar@{=}[d] & C^{}_{\mathcal{M}_Y/\mathcal{M}_X} \, \ar[d]^{\psi_0} \\
		\,\mathcal{M}_Y \, \ar@{^{(}->}[r] & \mathcal{M}^{}_{N_{Y/X}}.}
$$
The lower half is exactly the diagram (\ref{eq20}) and also commutes.
\end{proof}
\begin{proof}[Proof of Proposition \emph{\ref{prop3.1}}]
Proposition \ref{prop5.2} identifies $\psi_p$ in (\ref{eq5}) with the reduced Atiyah class $\Atp,$ which is injective by Proposition \ref{prop5.1}.
\end{proof}
\bibliographystyle{plain}\bibliography{name}
\end{document}